\documentclass[11pt,leqno]{amsart}

\usepackage{amsmath}
\usepackage{amssymb}
\usepackage{a4wide}
\usepackage{graphics}
\usepackage{epsfig}
\usepackage{enumerate}
\newtheorem{theorem}{Theorem}
\newtheorem{lemma}[theorem]{Lemma}
\newtheorem{proposition}[theorem]{Proposition}

\newtheorem{remark}[theorem]{Remark}

\newcounter{hypo}

\def\C{{\mathbb C}}

\def\N{{\mathbb N}} 
\def\R{{\mathbb R}}

\def\one{{\mathchoice {\rm 1\mskip-4mu l} {\rm 1\mskip-4mu l} {\rm 1\mskip-4.5mu l} {\rm 1\mskip-5mu l}}}
\def\re{\mathop{\rm Re}\nolimits}
 \def\im{\mathop{\rm Im}\nolimits}

\def\dive{\mathop{\rm div}\nolimits}

\def\ad{\mathop{\rm ad}\nolimits}
\def\<{\langle}
\def\>{\rangle}

\author[J.-F. Bony]{Jean-Fran\c{c}ois Bony}

\author[D. H\"{a}fner]{Dietrich H\"{a}fner}

\email{bony@math.u-bordeaux1.fr}
\email{hafner@math.u-bordeaux1.fr} 
\address{\newline Institut de Math\'ematiques de Bordeaux   \newline UMR 5251 du CNRS   \newline Universit\'e de Bordeaux I  \newline 351 cours de la Lib\'eration   \newline 33 405 Talence cedex    \newline France}

\title[Low frequency resolvent estimates]{Low frequency resolvent estimates for long range perturbations of the euclidean Laplacian}


\begin{document}

\begin{abstract}
Let $P$ be a long range metric perturbation of the Euclidean Laplacian on $\R^d,\, d\ge 3$. We prove that the following resolvent estimate holds:
\begin{equation*}
\Vert\<x\>^{-\alpha}(P-z)^{-1}\<x\>^{-\beta}\Vert\lesssim 1\quad \forall z\in \C\setminus \R,\,\vert z\vert<1,
\end{equation*}
if $\alpha,\beta> 1/2$ and $\alpha+\beta> 2$. The above estimate is false for the Euclidean Laplacian in dimension $3$ if $\alpha\le 1/2$ or $\beta\le 1/2$ or $\alpha+\beta<2$.
\end{abstract}

\maketitle

\section{Introduction}

There are now many results dealing with the low frequency behavior of the resolvent of Schr\"{o}dinger type operators. The methods used to obtain these results are various: one can apply the Fredholm theory to study perturbations by a potential (see {\it e.g.} \cite{JeKa}) or a short range metric (see {\it e.g.} \cite{Wa}). The resonance theory is also useful to treat compactly supported perturbations of the flat case (see  {\it e.g.} \cite{Bu}). Using the general Mourre theory, one can obtain limiting absorption principles at the thresholds (see {\it e.g.} \cite{FoSk} or \cite{Ri}). The pseudo-differential calculus of Melrose allows to describe the kernel of the resolvent at low energies for compactifiable manifolds (see  {\it e.g.} \cite{GH1}). Concerning the long range case, Bouclet \cite{Bo} has obtained a uniform control of the resolvent for perturbations in divergence form. We refer to his article and to \cite{DeSk} for a quite exhaustive list of previous results for perturbations of the Euclidean Laplacian.

On $\R^d$ with $d \geq 3$, we consider the following operator 
\begin{equation} \label{a5}
P= - b \dive ( G \nabla b ) = - \sum_{i,j=1}^{d} b(x) \frac{\partial \ }{\partial x_{i}} G_{i,j} (x) \frac{\partial \ }{\partial x_{j}} b (x) ,
\end{equation}
where $b(x)\in C^{\infty}(\R^d)$ and $G(x)\in C^{\infty}(\R^d;\R^{d\times d})$ is a real symmetric $d\times d$ matrix. The $C^{\infty}$ hypothesis is made mostly for convenience, much weaker regularity could actually be considered. We make an ellipticity assumption:
\begin{equation} \tag{H1} \label{a3}
\exists C>0, \ \forall x\in \R^d \qquad G(x)\ge C I_d \ \text{ and } \  b(x)\ge C, 
\end{equation}
$I_d$ being the identity matrix. We also assume that $P$ is a long range perturbation of the Euclidean Laplacian:
\begin{equation} \tag{H2} \label{a4}
\exists \rho > 0 , \ \forall \alpha \in \N^d \qquad \vert \partial^{\alpha}_x(G(x)-I_d) \vert + \vert \partial^{\alpha}_x ( b(x) - 1 ) \vert\lesssim \<x\>^{-\rho-\vert\alpha\vert}. \\
\end{equation}

In particular, if $b=1$, we are concerned with an elliptic operator in divergence form $P = - \dive (G \nabla )$. On the other hand, if $G=(g^2 g^{i,j}(x))_{i,j},\, b=(\det g^{i,j})^{1/4},\, g=\frac{1}{b}$, then the above operator is unitarily equivalent to the Laplace--Beltrami $- \Delta_{\mathfrak{g}}$ on $(\R^d, \mathfrak{g})$ with metric
\[\mathfrak{g} = \sum_{i,j=1}^{d} g_{i,j} (x) \, d x^i \, d x^j ,\]
where $(g_{i,j})_{i,j}$ is inverse to $(g^{i,j})_{i,j}$ and the unitary transform is just multiplication by $g$.  


\begin{theorem}\sl \label{th1}
Let $P$ be of the form \eqref{a5} in $\R^{d}$ with $d \geq 3$. Assume \eqref{a3} and \eqref{a4}.

$i)$ For all $\varepsilon>0$, we have
\begin{equation} \label{b1}
\big\Vert\<x\>^{-1/2-\varepsilon}(\sqrt{P}-z)^{-1}\<x\>^{-1/2-\varepsilon} \big\Vert \lesssim 1 ,
\end{equation}
uniformly in $z\in \C\setminus \R$, $\vert z\vert<1$.

$ii)$ For all $\varepsilon>0$, we have
\begin{equation} \label{b2}
\big\Vert \<x\>^{-1/2-\varepsilon}(P-z)^{-1}\<x\>^{-1/2-\varepsilon} \big\Vert \lesssim \vert z\vert^{-1/2} ,
\end{equation}
uniformly in $z\in \C\setminus \R$, $\vert z\vert<1$.

$iii)$ For all $\alpha,\beta> 1/2$ with $\alpha+\beta> 2$, we have
\begin{equation} \label{b3}
\big\Vert \<x\>^{-\alpha}(P-z)^{-1}\<x\>^{-\beta} \big\Vert \lesssim 1 ,
\end{equation}
uniformly in $z\in \C\setminus \R$, $\vert z\vert<1$.
\end{theorem}

\begin{remark}\sl \label{rem1}
$i)$ The estimate \eqref{b3} is not far from optimal. Indeed, this estimate is false for the Euclidean Laplacian $-\Delta$ in dimension $3$ if $\alpha\le 1/2$ or $\beta\le 1/2$ or $\alpha+\beta<2$.

$ii)$ One can interpret \eqref{b3} in the following way: one needs a $\< x \>^{-1/2}$ on the left and on the right to assure that the resolvent is continuous on $L^{2} (\R^{d} )$ and one needs an additional $\< x \>^{-1}$ (distributed, as we want, among the left and the right) to guarantee that its norm is uniform with respect to $z$.

$iii)$ By interpolation of \eqref{b2} and \eqref{b3}, for $\alpha,\beta > 1 /2$ with $\alpha+\beta \leq 2$, one obtains estimates like \eqref{b3} with $\vert z \vert^{- 1 + \frac{\alpha + \beta}{2} - \varepsilon}$ on the right hand side.

$iv)$ In dimension $1$, the kernel of $( - \Delta - z )^{-1}$ is given by $\frac{i e^{i \sqrt{z} \vert x - y \vert}}{2 \sqrt{z}}$. In particular, this operator satisfies \eqref{b2} but not \eqref{b3} (for any $\alpha , \beta$). Therefore it seems that \eqref{b2} is more general than \eqref{b3}. It could perhaps be possible to prove \eqref{b2} in lower dimensions (at least, in dimension $2$ and when $P$ is of divergence form $P = - \dive (G \nabla )$).

$v)$ For large $z$, the estimate \eqref{b2} coincides with the high energy estimate in the non-trapping case. In particular, if we suppose in addition a non trapping condition for $P$, then \eqref{b1} and \eqref{b2} hold uniformly in $z \in \C \setminus \R$.
\end{remark}


The proof of the above theorem is based on the low frequency estimates of \cite{BoHa}. Concerning the square root of $P$, they are used to treat the wave equation. Note that in \cite{BoHa} they are formulated for the Laplace--Beltrami operator $- \Delta_{\mathfrak{g}}$, but they obviously hold for the operators studied in the present paper. Essentially, we will show that \eqref{b1}$\Rightarrow$\eqref{b2}$\Rightarrow$\eqref{b3}.




\section{Proof of the results}

We begin by recalling some results of \cite{BoHa}. For $\lambda \geq 1$, we set
\begin{equation*}
{\mathcal A}_{\lambda} = \varphi ( \lambda P ) A_{0} \varphi ( \lambda P) ,
\end{equation*}
where
\begin{equation*}
A_{0}= \frac{1}{2} ( x D + D x ) , \quad D ( A_0 ) = \big\{ u \in L^{2} ( \R^{d} ) ; \ A_0 u \in L^{2} ( \R^{d} ) \big\} ,
\end{equation*}
is the generator of dilations and $\varphi \in C^{\infty}_{0} ( ] 0 , + \infty [ ; [0, + \infty [)$ satisfies $\varphi (x) > 1$ on some open bounded interval $I=[1-\widetilde{\varepsilon},1+\widetilde{\varepsilon}],\, 0<\widetilde{\varepsilon}<1$ sufficiently small. As usual, we define the multi-commutators $\ad_{A}^{j} B$ inductively by $\ad_{A}^{0} B =B$ and $\ad_{A}^{j+1} B = [ A , \ad_{A}^{j} B ]$.
We recall \cite[Proposition 3.1]{BoHa}:

\begin{proposition}\sl  \label{PM1}
$i)$ We have $(\lambda P)^{1/2}\in C^2({\mathcal A}_{\lambda})$. The commutators $\ad^j_{{\mathcal A}_{\lambda}}(\lambda P)^{1/2}$, $j=1,2$, can be extended to bounded operators and we have, uniformly in $\lambda \geq 1$,
\begin{align*}
\big\Vert \big[ {\mathcal A}_{\lambda} , ( \lambda P)^{1/2} \big] \big\Vert & \lesssim 1 ,  \\
\big\Vert \ad^2_{{\mathcal A}_{\lambda}} ( \lambda P)^{1/2} \big\Vert & \lesssim
\left\{\begin{aligned}
&1 && \rho > 1 , \\
&\lambda^{\delta} && \rho \leq 1 ,
\end{aligned} \right.
\end{align*}
where $\delta >0$ can be chosen arbitrary small. 

$ii)$ For $\lambda$ large enough, we have the following Mourre estimate:
\begin{equation*}
\one_{I} ( \lambda P) \big[ i ( \lambda P)^{1/2} , {\mathcal A}_{\lambda} \big] \one_{I} ( \lambda P) \geq \frac{\sqrt{\inf I}}{2} \one_{I} ( \lambda P ) .
\end{equation*}

$iii)$ For $0 \leq \mu \leq 1$ and $\psi \in C^{\infty}_{0} ( ] 0 , + \infty [ )$, we have
\begin{equation*}
\big\Vert \< {\mathcal A}_{\lambda} \>^{\mu} \psi ( \lambda P) \< x \>^{- \mu} \big\Vert \lesssim \lambda^{- \mu /2 + \delta} ,
\end{equation*}
for all $\delta >0$.
\end{proposition}
We will also need \cite[Lemma~B.12]{BoHa}:

\begin{lemma}\sl \label{Lem2}
Let $\chi \in C^{\infty}_{0} ( \R )$ and $\beta , \gamma \geq 0$ with $\gamma + \beta /2 \leq d/4$. Then, for all $\delta > 0$, we have
\begin{gather*}
\big\Vert \< x \>^{\beta} \chi ( \lambda P ) u \big\Vert \lesssim \lambda^{- \gamma + \delta} \big\Vert \< x \>^{\beta + 2 \gamma} u \big\Vert
\end{gather*}
uniformly in $\lambda\ge 1$.
\end{lemma}

By Mourre theory (see Theorem~2.2 and Remark~2.3 of \cite{BoHa} for example) and Proposition~\ref{PM1}, we obtain the following limiting absorption principle:
\begin{eqnarray}
\label{*}
\sup_{\re z\in I, \, \im z\neq 0} \big\Vert \<{\mathcal A}_{\lambda} \>^{-1/2-\varepsilon}((\lambda P)^{1/2}-z)^{-1}\<{\mathcal A}_{\lambda} \>^{-1/2-\varepsilon} \big\Vert \lesssim \lambda^{\delta} ,
\end{eqnarray}
for all $\varepsilon , \delta > 0$. This entails the following

\begin{lemma}\sl \label{Lem1}
For $\Psi\in C_0^{\infty}(]0, + \infty[)$ and $\varepsilon >0$, we have
\begin{gather}
\big\Vert \< x \>^{-1/2 - \varepsilon} \Psi ( \lambda P)( \sqrt{P} -\lambda^{-1/2}z)^{-1} \<x\>^{-1/2-\varepsilon} \big\Vert \lesssim 1 ,  \label{a1} \\
\big\Vert \<x\>^{-1/2-\varepsilon}\Psi(\lambda P)(P-\lambda^{-1}z^2)^{-1}\<x\>^{-1/2-\varepsilon} \big\Vert \lesssim \frac{\sqrt{\lambda}}{\vert z\vert} ,  \label{a2}
\end{gather}
uniformly in $\lambda\ge 1$ and $z\in \C\setminus \R$ with $\re z \in I$.
\end{lemma}

\begin{proof}
Let $\widetilde{\Psi}\in C_0^{\infty}(]0, + \infty[)$ be such that $\Psi\widetilde{\Psi}=\Psi$.

To prove the first identity, we write
\begin{align*}
\big\Vert\<x \>^{-1/2-\varepsilon} & \Psi(\lambda P)(P^{1/2}-\lambda^{-1/2}z)^{-1}\<x\>^{-1/2-\varepsilon} \big\Vert  \\
\lesssim& \big\Vert\<x\>^{-1/2-\varepsilon}\Psi(\lambda P)\<{\mathcal A}_{\lambda} \>^{1/2+\varepsilon} \big\Vert \big\Vert \<{\mathcal A}_{\lambda}\>^{-1/2-\varepsilon}(P^{1/2}-\lambda^{-1/2}z)^{-1}\<{\mathcal A}_{\lambda}\>^{-1/2-\varepsilon} \big\Vert   \\
&\times \big\Vert \<{\mathcal A}_{\lambda} \>^{1/2+\varepsilon}\widetilde{\Psi}(\lambda P)\<x\>^{-1/2-\varepsilon} \big\Vert   \\
\lesssim& \lambda^{- \frac{1}{4} - \frac{\varepsilon}{2} + \delta} \lambda^{\frac{1}{2} + \delta} \lambda^{- \frac{1}{4} - \frac{\varepsilon}{2} + \delta} \lesssim 1.
\end{align*}
Here we have used Proposition~\ref{PM1} $iii)$, Lemma~\ref{Lem2} as well as the fact that $\delta$ can be chosen arbitrary small.

To obtain \eqref{a2}, it is sufficient to write
\begin{align*}
\big\Vert \<x\>^{-1/2-\varepsilon} & \Psi(\lambda P)(P-\lambda^{-1}z^2)^{-1}\<x\>^{-1/2-\varepsilon}\big\Vert  \\
\lesssim& \lambda^{1/2}\big\Vert\<x\>^{-1/2-\varepsilon}\Psi(\lambda P)((\lambda P)^{1/2}+z)^{-1}\<x\>^{1/2+\varepsilon/2}\big\Vert\\
&\times \big\Vert\<x\>^{-1/2-\varepsilon/2}\widetilde{\Psi}(\lambda P)(P^{1/2}-\lambda^{-1/2}z)^{-1}\<x\>^{-1/2-\varepsilon}\big\Vert\\
\lesssim& \frac{\lambda^{1/2}}{\vert z\vert} .
\end{align*}
Here we have used \eqref{a1} and Lemma \ref{Lem2}. It is clear from the proof of Lemma \ref{Lem2} in \cite{BoHa} that we can apply it to $\Psi(\lambda P)((\lambda P)^{1/2}+z)^{-1}$ and that we gain $\frac{1}{\vert z\vert}$. Indeed, as an almost analytic extension, we can just take the almost analytic extension of $\Psi$ multiplied by the analytic function $\frac{1}{\sqrt{x}+z}$.  
\end{proof}

\begin{proof}[Proof of Theorem \ref{th1}]
We only show the third part of the theorem, the proof of the other parts is analogous. Also it is clearly sufficient to replace $z$ by $\lambda^{-1} \widetilde{z}^2$ with $\re \widetilde{z} =1 \in I$ and $\lambda \geq 1$ (for instance, $\lambda = ( \re \sqrt{z} )^{-2}$ and $\widetilde{z} = \sqrt{z}/ (\re \sqrt{z} )$). Let $\varphi , \widetilde{\varphi} \in C_0^{\infty} ([\frac{1}{3}, 3])$ and $f \in C^{\infty} ( \R )$ be such that $\widetilde{\varphi} =1$ on the support of $\varphi$, $f (x) =0$ for $x < 2$ and
\begin{equation*}
f(x) + \sum_{\mu=2^n,\, n\ge 0}\varphi(\mu x)=1 ,
\end{equation*}
for all $x >0$. Since $0$ is not an eigenvalue of $P$, we can write
\begin{equation*}
\<x\>^{-\alpha}(P- z)^{-1}\<x\>^{-\beta} = \<x\>^{-\alpha} f (P) (P-z)^{-1} \<x\>^{-\beta} + \sum_{\mu=2^n,\, n\ge 0}\<x\>^{-\alpha}\varphi(\mu P)(P-\lambda^{-1}\widetilde{z}^2)^{-1}\<x\>^{-\beta}.
\end{equation*}
Of course, since $\vert z \vert <1$, the functional calculus gives
\begin{equation*}
\big\Vert \<x\>^{-\alpha} f (P) (P-z)^{-1} \<x\>^{-\beta} \big\Vert \lesssim 1.
\end{equation*}
Let $\widetilde{\alpha} = \min ( \alpha,\frac{d}{2} )$ and $\widetilde{\beta}= \min ( \beta,\frac{d}{2} )$. Note that $\widetilde{\alpha}+\widetilde{\beta}>2$ since $d\geq 3$. Let $\Psi \in C^{\infty}_{0} ( ] 0 , + \infty [)$ be such that $\Psi =1$ near $[\frac{1}{12} , 12 ]$. Then, for $\frac{\mu}{4}\le \lambda\le 4\mu$, we have
\begin{align*}
\big\Vert & \< x \>^{-\alpha} \varphi ( \mu P) (P-\lambda^{-1} \widetilde{z}^2 )^{-1} \< x \>^{-\beta} \big\Vert  \\
&\lesssim \big\Vert \<x\>^{-\alpha} \varphi ( \mu P) \< x \>^{1/2 + \varepsilon} \big\Vert \big\Vert \<x\>^{-1/2-\varepsilon} \Psi ( \lambda P) (P-\lambda^{-1} \widetilde{z}^2 )^{-1} \< x \>^{-1/2-\varepsilon} \big\Vert \big\Vert \< x \>^{1/2 + \varepsilon} \widetilde{\varphi} ( \mu P) \<x\>^{-\beta} \big\Vert  \\
&\lesssim \lambda^{\frac{1}{4} + \frac{\varepsilon}{2} - \frac{\widetilde{\alpha}}{2} + \delta} \lambda^{\frac{1}{2}} \vert \widetilde{z} \vert^{-1} \lambda^{\frac{1}{4} + \frac{\varepsilon}{2} - \frac{\widetilde{\beta}}{2} + \delta} \lesssim \lambda^{1 + \varepsilon + 2 \delta - \frac{\widetilde{\alpha} + \widetilde{\beta}}{2}} \lesssim 1 ,
\end{align*}
for all $\varepsilon , \delta >0$ small enough. Here we have used \eqref{a2} and two times Lemma~\ref{Lem2}. On the other hand, for $\lambda\notin [\frac{\mu}{4}, 4\mu]$, the functional calculus and Lemma~\ref{Lem2} yield
\begin{align*}
\big\Vert \<x\>^{-\alpha} \varphi(\mu P)(P-\lambda^{-1}\widetilde{z}^2)^{-1}\<x\>^{-\beta}\big\Vert & \lesssim \vert\mu^{-1}-\lambda^{-1}\vert^{-1}\big\Vert\<x\>^{-\alpha}\varphi(\mu P)\big\Vert\big\Vert \widetilde{\varphi}(\mu P) \<x\>^{-\beta} \big\Vert\\
&\lesssim \vert\mu^{-1}-\lambda^{-1}\vert^{-1}\mu^{-\frac{\widetilde{\alpha}+\widetilde{\beta}}{2} + \varepsilon} ,
\end{align*}
for all $\varepsilon >0$. Splitting the sum into two, we get
\begin{gather*}
\sum_{4 \mu < \lambda} \vert \mu^{-1}-\lambda^{-1} \vert^{-1} \mu^{-\frac{\widetilde{\alpha}+\widetilde{\beta}}{2} + \varepsilon} \lesssim \sum_{4 \mu<\lambda} \mu \mu^{-\frac{\widetilde{\alpha}+\widetilde{\beta}}{2} + \varepsilon} \lesssim 1,\\
\sum_{\mu > 4 \lambda}\vert\mu^{-1}-\lambda^{-1}\vert^{-1}\mu^{-\frac{\widetilde{\alpha}+\widetilde{\beta}}{2}+ \varepsilon} \lesssim \sum_{\mu > 4 \lambda}\lambda\mu^{-\frac{\widetilde{\alpha}+\widetilde{\beta}}{2} + \varepsilon}\lesssim 1.
\end{gather*}
This finishes the proof of the theorem.
\end{proof}

\begin{proof}[Proof of Remark~\ref{rem1} $i)$]
Let us recall that the kernel of the resolvent of the flat Laplacian in $\R^{3}$ at $z=0$ is given by
\begin{equation*}
K(x,y,0)=\frac{1}{4\pi \vert x-y\vert}.
\end{equation*}
Assume that $\<x\>^{-\alpha}(-\Delta)^{-1}\<x\>^{-\beta}$ is bounded on $L^{2} (\R^{3})$. Applying to $\chi\in C_0^{\infty}(\R^3)\subset L^2(\R^3)$, we find
\begin{equation*}
\big( \<x\>^{-\alpha}(-\Delta)^{-1}\<x\>^{-\beta}\chi \big) (x)=\int\frac{1}{4\pi \vert x-y\vert}\<x\>^{-\alpha}\<y\>^{-\beta}\chi(y)dy\gtrsim \<x\>^{-\alpha-1} ,
\end{equation*}
for $\vert x \vert \gg 1$. But $\<x\>^{-1-\alpha}\in L^2 (\R^{3} )$ if and only if $\alpha>1/2$. The condition $\beta> 1/2$ is checked in the same way.
We now apply the resolvent to $f(x)=\<x\>^{-3/2-\varepsilon}\in L^2(\R^3)$ and find
\begin{align*}
\big( \<x\>^{-\alpha}(-\Delta)^{-1}\<x\>^{-\beta}f \big)(x) & = \int\frac{1}{4\pi \vert x-y\vert}\<x\>^{-\alpha}\<y\>^{-\beta}
\<y\>^{-3/2-\varepsilon}dy\\
&\geq \int_{\vert y\vert\le \frac{\vert x\vert}{2}}\frac{1}{4\pi \vert x-y\vert}\<x\>^{-\alpha}\<y\>^{-\beta}\<y\>^{-3/2-\varepsilon}dy\\
&\gtrsim \<x\>^{-\alpha-1}\int_{\vert y\vert\le \frac{\vert x\vert}{2}}\<y\>^{-3/2-\varepsilon-\beta}dy\gtrsim\<x\>^{3/2-\alpha-\beta-1-\varepsilon}.
\end{align*}
This leads to the condition $2 ( 3/2 - \alpha - \beta- 1 ) \leq -3$ which implies $\alpha + \beta \geq 2$.
\end{proof}

\end{document}